\documentclass[10pt]{amsart}

\usepackage{amssymb, amsthm, amsmath, graphicx, yfonts, wasysym, appendix, url, verbatim}
\usepackage[margin=4cm]{geometry}

\newcommand{\IP}[2]{\left< #1 , #2 \right>}

\newcommand{\arccosh}{\text{arccosh}}
\newcommand{\Q}{\ensuremath{\mathcal{Q{}}}}

\newcommand{\R}{\ensuremath{\mathbb{R}}}
\newcommand{\D}{\text{D}}

\renewcommand{\S}{\ensuremath{\mathbb{S}}}

\newcommand{\gmcfh}{GMCF${}^{H \le\, 0}$}

\allowdisplaybreaks[4]

\newtheorem{thm}{Theorem}[section]
\newtheorem{cor}[thm]{Corollary}
\newtheorem{prop}[thm]{Proposition}
\newtheorem{lem}[thm]{Lemma}

\newtheorem*{defn}{Definition}

\theoremstyle{remark}
\newtheorem*{rmk}{Remark}

\title[Mean curvature flow with free boundary outside a hypersphere]{Mean curvature flow with free boundary outside a hypersphere}
\author{Glen Wheeler and Valentina-Mira Wheeler$^*$}
\thanks{*: Corresponding author.}
\address{Valentina-Mira Wheeler \\
           Institut f\"ur Mathematik,
           Universit\"at Potsdam\\
           Am Neuen Palais 10,\\
           D-14469 Potsdam, Germany\\
           email: vulcanov@math.zedat.fu-berlin.de }
\address{\emph{Current: }Glen Wheeler, Valentina-Mira Wheeler \\
           Institute for Mathematics and its Applications \\
           University of Wollongong\\
           Northfields Avenue,\\
           Wollongong, NSW, 2522, Australia\\
           email: vwheeler@uow.edu.au \\
           email: glenw@uow.edu.au }


\begin{document}

\begin{abstract}
The purpose of this paper is twofold: firstly, to establish sufficient conditions
under which the mean curvature flow supported on a hypersphere with exterior
Dirichlet boundary exists globally in time and converges to a minimal
surface, and secondly, to illustrate the application of Killing vector fields in
the preservation of graphicality for the mean curvature flow with free
boundary.
To this end we focus on the mean curvature flow of a topological annulus with
inner boundary meeting a standard $n$-sphere in $\R^{n+1}$ perpendicularly and
outer boundary fixed to an $n-1$-sphere with radius $R>1$ translated by a
vector $he_{n+1}$ for $h\in\R$ where $\{e_i\}_{i=1,\ldots,n+1}$ is the standard
basis of $\R^{n+1}$.
We call this the \emph{sphere problem}.
Our work is set in the context of graphical mean curvature flow with either
symmetry or mean concavity/convexity restrictions.
For rotationally symmetric initial data we obtain, depending on the exact
configuration of the initial graph, either long time existence and convergence
to a minimal hypersurface with boundary or the development of a finite-time
curvature singularity.
With reflectively symmetric initial data we are able to use Killing vector
fields to preserve graphicality of the flow and uniformly bound the mean
curvature pointwise along the flow.
Finally we prove that the mean curvature flow of an initially mean
concave/convex graphical surface exists globally in time and converges to a
piece of a minimal surface.
\end{abstract}

\maketitle

\section{Introduction}

There has been much work on the mean curvature flow problem for immersions and
graphs with or without boundary conditions.
The study of Ecker--Huisken \cite{ecker1989mce,ecker1991ieh} is a seminal work
including a sharp theorem on global existence for initially graphical Lipschitz
data.
The non-parametric mean curvature flow of graphs with either a ninety-degree
contact angle or Dirichlet boundary condition on cylindrical domains has been
studied by Huisken \cite{huisken1989npm}, who proved a global existence
theorem.
In this direction we also mention the work of Altschuler--Wu
\cite{altschuler1994} which allows for arbitrary contact angle at the boundary
for graphs over $\R^2$.
Guan \cite{guan1996mean} later generalised this to arbitrary intrinsic
dimension.
Recently Shahriyari \cite{leithesis} proved that any complete translating
solution to the mean curvature flow in $\R^3$ must be either wedged between two
planes, one one side of a plane, or entire.
This and other results from \cite{leithesis} shed new light on \cite{altschuler1994,guan1996mean}.

A natural next step in this line of research is to study the mean curvature
flow of graphs with a free boundary on a fixed hypersurface in $\R^{n+1}$.
This began with a series of results on the mean curvature flow of immersions
with free boundary, where a restriction on the angle of contact with a fixed
hypersurface in Euclidean space is imposed.
In \cite{thesisstahl} Stahl proved that the mean curvature flow with free
boundary on a fixed support hypersurface $\Sigma$ either exists for all time or
develop a curvature singularity.
In the special case of convex initial hypersurfaces and convex umbilic support
hypersurfaces $\Sigma$ he proved that the curvature becomes unbounded in finite
time and that the rescaled solution is asymptotic to a hemisphere.

Buckland \cite{buckland2005mcf}, using a localised reflection technique, proved
a monotonicity formula for mean curvature flow with a free boundary analogous
to the groundbreaking result of Huisken \cite{huisken1990asymptotic}.
See also \cite{ecker2001local} for a local version of Huisken's monotonicity
formula.
Buckland's result provides, again in the case of umbilic, convex contact
hypersurfaces, with a classification of Type I singularities on the boundary.
Regularity theory for this problem has been developed by Koeller
\cite{koeller2007singular} using the reflection construction of Buckland and
the local boundary estimates of Stahl.
He obtained results analogous to those for the compact mean curvature flow, see
\cite{ecker2004rtm}, in the boundary setting.

In this paper we consider the mean curvature flow of graphs with a free
boundary on $\S^{n} \subset \R^{n+1}$ anchored at a fixed Dirichlet height
outside the sphere.
Stahl's earlier investigation into the mean curvature flow
\cite{stahl1996convergence,stahl1996res} treats the problem on the interior of
a sphere.
The picture to keep in mind, for Stahl's result, is of a bubble evolving on the
inside of a fixed sphere.
One of the primary motivations for our work here is to treat the problem on the
exterior of a fixed sphere, complementing the results of Stahl.

In particular, suppose $\Sigma$ is a standard $n$-sphere in $\R^{n+1}$ with
$n\geq 2$ centred at the origin.
We use $\nu^\Sigma:\R^{n+1}\rightarrow\R^{n+1}$ to denote its unit inner normal
vectorfield.
Let $M^n$ be a smooth, orientable $n$-dimensional Hausdorff paracompact
manifold with two smooth, disjoint boundaries $\partial_N M^n$ and $\partial_D
M^n$, where the subscripts $N$ and $D$ stand for Neumann and
Dirichlet respectively.
Set $M_0:=F_0(M^n)\subset{\R}^{n+1}$ where $F_0:M^n\rightarrow {\R}^{n+1}$ is a
smooth embedding satisfying
\begin{align*}
&{\partial}_N M_0 \ \equiv \ F_0({\partial}_N M^n)\ =\ M_0 \cap \Sigma,\\
&\IP{{\nu}^{M_0}}{{\nu}^{\Sigma} \circ F_0} (p)\ =\ 0, \text{ }\forall
p~\in
{\partial}_N M^n,\\
&{\partial}_D M_0 \ \equiv\  F_0({\partial}_D M^n)\ =\ \partial
B_R\big(O_n^{h_0}\big),
\end{align*}
for some positive $R\ >\ 1$.
We use the embedding $F_0$ to induce a Riemannian structure on $M^n$ via the
pullback of the Euclidean metric; that is, $(M^n, F_0^*\delta)$ is a Riemannian
manifold where $\delta$ is the standard metric on $\R^n$.
In the above we denoted by $\partial B_R\big(O_n^{h_0}\big)$ the boundary of
the $n$-dimensional disk $B_R\big(O_{n}^{h_0}\big)$ of radius $R$ centred at
the origin $O_n^{h_0}=(0,0,..,0,h_0)$ in the hyperplane $\{x_{n+1}=h_0\}$.

We use the initial data above to generate a mean curvature flow with boundary.
Let $I \subset {\R}$ be an open interval and $F_t=F(\cdot,t):M^n \rightarrow
{\R}^{3}$ be a one-parameter family of smooth embeddings for all $t\in I$.
The family of hypersurfaces $(M_t)_{t\in I}$, where $M_t=F_t(M^n)$, are said to
be evolving by mean curvature flow with Neumann free boundary condition on
$\Sigma$ and a constant $h_0$ height on the Dirichlet boundary if
\begin{align}
&\frac{\partial F}{\partial t}(p,t)\ =\ -\  H(p,t)\nu^{M_t}(p,t),\quad&&\forall (p,t)\in M^n \times I,\text{ flow equation,}
\notag
\\
&F(p,0)\ =\ F_0(p),&&\forall p\in M^n\text{, initial condition,}\notag\\
&F(p,t)\  \in \ \Sigma, &&\forall (p,t)\in {\partial}_N M^n \times I\text{, contact condition,}\label{immersion}\\*
&\IP{\nu^{M_t}}{{\nu}^{\Sigma} \circ F} (p,t)\ =\ 0, &&\forall (p,t)\in {\partial}_N M^n \times I\text{, Neumann condition,}\notag\\
&F(p,t)\ =\ F_0(p) \in \partial B_R(O_n^{h_0}),&&\forall (p,t)\in {\partial}_{D}M^n \times I\notag\text{, Dirichlet condition,}\\*
&H(p,0)\ =\ 0,&&\forall (p,t)\in {\partial}_{D}M^n \notag\text{, compatibility condition.}
\end{align}
Our convention is, throughout this work and when not stated otherwise, that the unit normal ${\nu}^{\Sigma}$ to $\Sigma$ points outside the evolving surfaces $M_t$.
Here this means that it points into the sphere, making (with our sign conventions) the curvature of $\Sigma$ negative.

The first issue to be treated is the short time existence of the problem, independent of an additional graph condition.
This can be easily obtained if we write the surfaces for a short time over the initial manifold and apply standard parabolic theory such as is contained in \cite{lieberman1996second}.
This yields the following theorem.

\begin{thm}[Short time existence]
Suppose $F_0:M^n\rightarrow\R^{n+1}$ is as above.
There exists a maximal $T > 0$ such that a one-parameter family of embeddings $F:M^n\times[0,T)\rightarrow\R^{n+1}$ satisfying
\eqref{immersion} exists, where for each $t\in[0,T)$ the embedding $F_t$ is smooth.
The family of embeddings $F$ is unique up to reparametrisation.
\label{thmshortimeexsitence}
\end{thm}

A detailed exposition of the proof of this result can be found in \cite{thesisvulcanov}.

We typically consider the problem \eqref{immersion} under the additional assumption that the initial hypersurface is also a graph in the direction of a fixed vector field $\zeta$ in ${\R}^{n+1}$, that is
\begin{align}
\IP{\nu^{M_0}}{\zeta}>0. \label{graphcondition}
\end{align}
We are interested in the long time behaviour of \eqref{immersion} with or without the graph condition \eqref{graphcondition}.

The most restrictive setting which we consider is that of initially rotationally symmetric graphs.
In this case, the problem \eqref{immersion} is described by a family of functions $\omega:(r(t),R)\times [0,T) \rightarrow \R$ evolving by
\begin{align}
\frac{\partial \omega}{\partial t}\ \  &= \frac{d^2\omega}{dy^2}\
\frac{1}{1+(\frac{d\omega}{dy})^2}+\frac{d\omega}{dy}\
\frac{n-1}{y}~~\text{ on }~~ \bigcup_{t\in[0,T)}(r(t),R) \times \{t\},\label{sphereradiallysymmetricgraph}\\
\frac{d \omega}{dy}(r(t),t) \ &= \ \frac{\sqrt{1-r(t)^2}}{r(t)}
\text{ and }{\omega}^2(r(t),t)+r^2(t)\ =\ 1~~\text{ for all
}~~t \in [0,T),\notag\\
\omega(R,t) \ &=\ h_0 ~~\text{ on }~~[0,T)\notag,\\
\omega(y,0)\  &= \ \omega_0~~\text{ on }~~ (r(0),R)\notag,
\end{align}
Here we have used the fact that $\Sigma$ is a unit sphere centred at the origin of ${\R}^{n+1}$, specifically that $|{\omega}_{\Sigma}|\ =\ \sqrt{1-y^2}$ and ${\nu}^{\Sigma}=-\sqrt{1-y^2}\left(\frac{y}{\sqrt{1-y^2}},\
1\right)$ if above the $\R^n$ plane or the opposite sign otherwise, where $y=|(x_1,\ldots,x_n)|_{{\R}^n}$ and $\omega_\Sigma$ is the graph that generates $\Sigma$. Also let us denote by $D(t)=(r(t),D)$ the domain of $\omega(\cdot,t)$ for all $t\geq 0$.

Even in the graphical rotationally symmetric setting one may encounter finite-time singularities.
This can be proved by pinching the evolving family at the North or South pole of the supporting sphere.
Some additional arguments allow us to further pin-down the rate at which the second fundamental form blows up at this singularity.
The precise result is the following.

\begin{thm}[Curvature singularity on the boundary]
Let $\omega$ satisfy \eqref{sphereradiallysymmetricgraph} with $\sup_{D(0)} |\omega_0|\ > \ 1$.
If there exists a self-similar torus in the region of $\R^{n+1}$ defined by $\{ (x_1,\ldots,x_{n+1})\ :\ 1<|x_{n+1}|< \sup_ {D(0)}|\omega_0(y)|\}$
then the solution for the problem \eqref{sphereradiallysymmetricgraph} exists for at most finite time $T<\infty$ and the graphs $\omega(y,t)$ develop a curvature singularity at $y=0$ as $t\rightarrow T$.
Furthermore there exists a positive constant $C<\infty$ such that
\begin{align*}
||A||_\infty^2(y)\leq \ C\frac{1}{(T-t)^2},
\end{align*}
where we have denoted by $A$ the second fundamental form of the hypersurfaces evolving by mean curvature flow generated by $\omega$.
\label{thmcurvaturesingularity}
\end{thm}

Pinching off at the North pole or South pole is in fact the only kind of singularity that can occur in this setting.
Once we rule this out, there is no further obstacle to global existence.
This was treated for general rotationally symmetric support hypersurfaces $\Sigma$ in \cite{vmwheeler2012rotsym}.
The following result is a strengthening of the global existence theorem from \cite{vmwheeler2012rotsym} in the special case where the support surface $\Sigma$ is a sphere.
The improvement follows by using pieces of catenoids as comparison hypersurfaces.

\begin{thm}[Global existence]
\label{thmlongtimeexistencecatenoid}
Let ${\omega}_0:(r_0,R)\rightarrow \R$ be a smooth function satisfying the boundary conditions in \eqref{sphereradiallysymmetricgraph}.
Suppose there exist constants $d_i$, $C_i$, ${\varepsilon}_i \in [0,1)$ and $y_i\in(0,r_0)$ for $i=1,2$ such that
\begin{align}
-d_1 \arccosh(C_1y)-{\varepsilon}_1 \ <\ {\omega}_0(y)\ <\ d_2
\arccosh(C_2y)+{\varepsilon}_2, \text{   }\forall
y\in[r_0,R],\label{initialheightcatenoid}
\end{align}
\begin{align} d_i \arccosh(C_iy_i)+{\varepsilon}_i\ =\
\sqrt{1-y_i^2}, \label{catenoidintersectionsphere}
\end{align}
and
\begin{align}
0\ \leq \ (1-d_i^2)c_i^2y_i^2-c_i^2y_i^4 -1+y_i^2\,.
\label{catenoidsphereangle}
\end{align}
Then there exists a global solution $\omega:(r_t,R)\times[0,\infty)\rightarrow\R$ to \eqref{sphereradiallysymmetricgraph} with initial data $\omega_0$.
Furthermore, as $t\rightarrow\infty$ the hypersurfaces generated by the rotation of $\omega$ converge to a piece of a minimal hypersurface.
In particular, if $h_0=0$ then the hypersurfaces converge to the flat annulus around the sphere $\Sigma$.
\end{thm}

Relaxing the continuous symmetry imposed by rotationally symmetric initial data to a discrete reflective symmetry causes additional difficulty through the possibility of boundary \emph{tilt}, where the gradient of the
graph becomes unbounded at the boundary with bounded curvature.
We show that while the curvature is bounded, that is for all $t<T$, initially reflectively symmetric graphs remain graphical under the mean curvature flow.

The key idea is to employ Killing vector fields of Euclidean space.
In order to state the result precisely we require some notation.
In $\R^{n+1}$ there are $n(n+1)/2$ Killing vector fields of rotation and $n+1$ translations from which $e_{n+1}$ is one.
We denote the first $n$ of the Killing vector fields of rotation by
\begin{align*}
K_i(x_1,\ldots,x_{n+1})
 = e_ix_{n+1} - x_ie_{n+1}
\end{align*}
for each $i=1,\ldots,n$.
Let us define $\xi$ to be the vector field tangential to the sphere that generates the vertical great circles passing through the North and South Poles.
That is
\begin{align}
\label{xidefn}
\xi(x_1,\ldots,x_{n+1})
 & = \  \bigg(-x_1x_{n+1},\ldots,-x_nx_{n+1},\sum_{i=1}^n x_i^2\bigg)\notag\\
 & = \ - \ \sum_{i=1}^{n} x_i K_i .
\end{align}
One can see that
\[
\IP{\nu^{M_t}}{\xi}=\sum_{i=1}^n-x_i\IP{\nu^{M_t}}{K_i}
\]
and so the graph condition \eqref{graphcondition} in $\xi$ follows from the following set of conditions
\begin{align}
&\IP{\nu^{M_t}}{K_i}>0 \text{ when }x_i<0,\notag\\
&\IP{\nu^{M_t}}{K_i}<0 \text{ when }x_i>0,\label{initialsphere}\\
&\IP{\nu^{M_t}}{K_i}=0 \text{ when }x_i=0,\notag
\end{align}
for all $i=1,\ldots,n$.
The strategy is to take our graph to be initially reflectively symmetric and impose the stronger set of conditions \eqref{initialsphere}, implying the positivity of $\IP{\nu^{M_t}}{\xi}$, and show that they are
preserved.

\begin{thm}[Preservation of graphicality]
\label{thmLongsphere}
Let $F_t$ satisfy \eqref{immersion} for $t\in[0,T)$ and be reflectively symmetric over the planes $\{x_i=0\}$ for all $i=1,\ldots, n$.
Suppose that the initial immersion $M_0=F_0(M^n)$ satisfies conditions \eqref{initialsphere}, is initially graphical \eqref{graphcondition} for $\zeta=e_{n+1}$, and that the height bound $|\IP{F_0}{e_{n+1}}|\leq 1$ is
satisfied.
Then
\begin{align*}
\IP{\nu^{M_t}}{e_{n+1}}\ > \ 0
\end{align*}
for all $t\in[0,T)$.
That is, the solution remains graphical for all times of existence.
\end{thm}

While we are able to control the mean curvature along the flow under the initial conditions in Theorem \ref{thmLongsphere}, we are not able to control the full second fundamental form, and this prevents us from obtaining
true global existence.
To finish the paper we present the following global existence theorem, which holds without any of the symmetry conditions imposed above.
Instead, we require the initial data to be graphical and have non-positive (or non-negative) mean curvature.
The precise statement is as follows.

\begin{thm}[Global existence for mean concave (mean convex) initial surfaces]
\label{thmSombrero}
Let $F_t = F(M^2,t)$ satisfy \eqref{immersion} for $t\in[0,T)$ with initially graphical mean concave (mean convex) data satisfying the height bound $0 < \IP{F_0}{e_{3}} < 1$ (or $-1 < \IP{F_0}{e_{3}} < 0)$).
Then the solution exists for all time and converges to a piece of a minimal surface.
\end{thm}

The catenoid comparison method in the proof of Theorem
\ref{thmlongtimeexistencecatenoid} allows one to weaken the initial height
bound in Theorems \ref{thmLongsphere} and \ref{thmSombrero} as mentioned in
Section 2 below.
For clarity of exposition we have used the more restrictive $0 <
\IP{F_0}{e_{3}} < 1$.

The paper is organised as follows.
Rotationally symmetric graphs are considered in Section 2, where Theorem 1.3 is proved.
Section 3 is concerned with the case of reflectively symmetric graphs.
There our goal is to prove Theorem 1.4.
Finally, in Section 4 we consider mean concave (mean convex) graphs, and prove Theorem 1.5.
We have attempted to make each section self-contained.

\section{Rotationally symmetric graphs}

Theorems on rotationally symmetric graphs outside general rotationally symmetric contact surfaces can be found in \cite{vmwheeler2012rotsym}.
Results particular to the case we consider here, where $\Sigma$ is the standard $n$-sphere, can be inferred from these more general results.
Here we only include the details which differ from the proofs found in \cite{vmwheeler2012rotsym}.
Long time existence for \eqref{sphereradiallysymmetricgraph} is obtained from uniform height and gradient bounds.
The latter follows from standard interior estimates and the rotational symmetry.

On the boundary, we need to exclude behaviour in which the evolving graphs reach points where the sphere has a horizontal point, so the North and South Pole.
Usually this is achieved by beginning the flow with a graph such that in the region between the maximal and minimal height value there is no point where $\Sigma$ is horizontal.
Then we preserve the height of the graphs for all times between the initial values.
The preservation part of the argument is still valid but here we allow initial data which have a height above (or below) the critical value $1$ (or $-1$).
To prove that the graphs do not move towards the North or South Pole of the sphere we apply the comparison principle with two pieces of a minimal surface.

\begin{proof}[Proof of Theorem \ref{thmlongtimeexistencecatenoid}.]
Suppose we are in the setting of Theorem \ref{thmlongtimeexistencecatenoid}.
By \eqref{initialheightcatenoid}, the initial graph is above and below two pieces of catenoids defined on $[y_i,\infty)$.
The catenoids touch the sphere below and above the initial graph respectively.
Condition \eqref{catenoidintersectionsphere} implies that the lower catenoid meets the sphere at $-\sqrt{1-y_1^2}$ and that the upper catenoid meets the sphere at $\sqrt{1-y_2^2}$.
The angle between the sphere and the two pieces is given by right hand side of \eqref{catenoidsphereangle}.
It follows from this relation that the angle is less than or equal to ninety degrees.
The initial graph starts between these two catenoids and due to the choice of angle at the intersection of the sphere with the two catenoids, the comparison principle shows that for all time the family of graphs remains
contained between the upper and lower catenoids.

In particular, if the catenoids meet the sphere at precisely ninety degrees then one can prove this by using the comparison principle for hypersurfaces evolving by mean curvature flow from Huisken (see
\cite{huisken1986cch}). Details of the modifications required for the free boundary setting can be found in \cite[Theorem 2.10]{thesisvulcanov}).
If the contact angle is strictly less than ninety degrees then the hypersurfaces will meet for the first time in the interior, a case excluded by Huisken's comparison principle \cite{huisken1986cch}.
\end{proof}

This theorem can be used to allow the initial graph to attains heights greater than that of the sphere.
We modestly demonstrate this with an explicit construction where the initial graph reaches $1$ or $-1$.

\begin{cor}
Let $\omega$ satisfy \eqref{sphereradiallysymmetricgraph} with $|{\omega}_0|\leq 1$.
Then Theorem \ref{thmlongtimeexistencecatenoid} is applicable.
\end{cor}
\begin{proof}
For the existence of the catenoids used as a barriers in the above theorem we have to first prove that the Neumann boundary of the initial graph is not equal to the North or South Pole of the sphere.
This is the same as proving that there exists a strictly positive constant ${\varepsilon}_i$ for the choice of catenoidal barriers in \eqref{initialheightcatenoid} and \eqref{catenoidintersectionsphere}.
Once we have $\varepsilon_i$, it is easy to choose the other constants $C_{i}$, $d_i$ and $y_{i}$.

Suppose that the initial graph satisfies ${\omega}_0(r_0)=1$, which implies that $r_0=0$.
Thus we find ourselves at the North pole of the sphere.
If the gradient of $\omega_0$ at the boundary is bounded or non-positive, then the Neumann condition is not satisfied there.
Therefore $\frac{d{\omega}_0}{dy}(r_0)=\frac{1}{r_0}=+\infty$.
This implies that there exists a $y\in(r_0,R)$ such that ${\omega}_0(y)>1$, which is a contradiction with the initial height bound.
A similar argument contradicts the assumption that ${\omega}_0(r_0)=-1$.

Thus there exist positive constants ${\varepsilon}_i\in [0,1)$ for $i=1,2$.
It is then straightforward to choose the rest of the constants which characterise the two catenoidal pieces found in Theorem \ref{thmlongtimeexistencecatenoid}.
\end{proof}

\begin{rmk}
Note that if we allow $|\omega_0| > 1$ then we must add additional restrictions, since if $\omega_0(r_0) = 1-\varepsilon$ for a sufficiently small $\varepsilon$ then a self-similar torus may be inserted underneath the graph $\omega_0$ close to the Neumann boundary which forces the singularity. The other issue is that if we allow arbitrary heights greater than one, then if the Dirichlet boundary is at a sufficiently large radius we may always place a self-similar torus under the initial graph. These are the only essential obstructions however; one may enforce an additional height restriction on the Neumann boundary and restrict the radius at the Dirichlet boundary in order to allow heights strictly arbitrarily large (in particular greater than 1) to be reached on the interior.
\end{rmk}
\section{Reflectively symmetric graphs}

In this section we consider initial hypersurfaces that satisfy \eqref{graphcondition} for $\zeta=e_{n+1}$ and condition \eqref{initialsphere}, which implies \eqref{graphcondition} holds with respect to two vectors: $e_{n+1}$ and $\xi$ defined in
\eqref{xidefn}.
We further assume that the initial data is reflectively symmetric over the hyperplanes $\{x_i=0\}$.

We collect these assumptions in the following definition.

\begin{defn}[RGMCF]
We say that $M_t = F(M^n,t)$ is a reflectively symmetric graphical mean curvature flow outside a sphere (RGMCF) if
\begin{enumerate}
\item[(i)]   $M_t = F(M^n,t)$ is a one-parameter family of hypersurfaces evolving by mean curvature flow outside a standard unit sphere in accordance with \eqref{immersion};
\item[(ii)]  $h_0 = 0$;
\item[(iii)] Condition \eqref{initialsphere} holds on $M_0$; and
\item[(iv)]  $M_0$ is reflectively symmetric across the hyperplanes $\{x_i=0\}$.
\end{enumerate}
\end{defn}

Our strategy is to use the initial reflective symmetry and the maximum principle on the evolution equations for $s_i:=\IP{\nu^{M_t}}{K_i}$ to show that the graph condition \eqref{graphcondition} with respect to both
$e_{n+1}$ and $\xi$ is preserved.
The results presented in this section are the generalisation to hypresurfaces, i.e. $n>2$, of the theorems for surfaces obtained in \cite{thesisvulcanov}.
There the terminology ``tilt point'' is introduced.
A tilt point is a point on the free boundary where we have lost the graph property in the $e_{n+1}$ direction.
At a tilt point the normal vector is horizontal.
Here let us extend this definition slightly as follows.

\begin{defn}[Tilt]
Let $X\in\R^n$ be a vector field and $M_t = F(M^n,t)$ be a one-parameter family of hypersurfaces evolving by mean curvature flow with at least one free boundary $\partial_NM^n$.
We call a point $(p,t) \in \partial_NM^n\times[0,T)$ a \emph{tilt in the $X$ direction} if $\IP{\nu^{M_t}(p,t)}{X(F(p))} = 0$.
\end{defn}

We first show that for $\Sigma$ an $n$-sphere, tilt in the $\xi$ direction and tilt in the $e_{n+1}$ direction are equivalent.

\begin{prop}
\label{propTilt}
Suppose $M_t = F(M^n,t)$ is a one-parameter family of hypersurfaces evolving by mean curvature flow outside a standard unit sphere in accordance with \eqref{immersion}.
On $\Sigma$ we have
\[
\IP{\nu^{M_t}}{\xi}=0\text{ if and only if }\IP{\nu^{M_t}}{e_{n+1}}=0.
\]
The same holds for any tangent vector to the sphere $\Sigma$ independent of the flow.
\end{prop}
\begin{proof}
The proof is basic and uses the fact that the position vector of a sphere is of constant length.
Suppose that $\IP{\nu^{M_t}}{\xi}=0$.
We want to show that $\IP{\nu^{M_t}}{e_{n+1}}=0$, i.e. $\nu_{n+1}=0$, where we have used the notation $\nu_i:=\IP{\nu^{M_t}}{e_i}$.
To show this we compute
\begin{align*}
0 &= \IP{\nu^{M_t}}{\xi}
\notag\\
  &= -x_{n+1}\sum_{i=1}^n x_i\nu_i + \nu_{n+1}\sum_{i=1}^n x_i^2\notag\\
  &= \nu_{n+1} x_{n+1}^2+\nu_{n+1}\big(1-x_{n+1}^2\big)\notag\\
  &= \nu_{n+1}\,,
\end{align*}
where we used that $\nu^{M_t}$ is tangent to $\Sigma$ and that on $\Sigma$ we have $\sum_{i=1}^n x_i^2 = 1$.
Conversely, the above computation also shows that $\IP{\nu^{M_t}}{\xi}=0$ if $\IP{\nu^{M_t}}{e_{n+1}}=0$.
\end{proof}

The following lemma shows that Dirichlet boundary conditions are consistent with conditions \eqref{initialsphere} so long as $h_0=0$.

\begin{lem}Suppose $M_t = F(M^n,t)$ is a one-parameter family of hypersurfaces evolving by mean curvature flow outside a standard unit sphere in accordance with \eqref{immersion} and $h_0=0$.
On the Dirichlet boundary ${\partial}_D M^n$ condition \eqref{graphcondition} for $\zeta=e_{n+1}$ implies \eqref{initialsphere} on $\partial_DM^n$.
\label{propInitialD}
\end{lem}
\begin{proof}
In the canonical orthonormal basis of $\R^{n+1}$ we compute at a point on the Dirichlet boundary where $\{x_{n+1}=0\}$:
\begin{align*}
\IP{\nu^{M_t}}{K_i}\ =\nu_1x_{n+1}-\nu_{n+1}x_i=-\nu_{n+1}x_i.
\end{align*}
Note that $\nu_{n+1}>0$ from \eqref{graphcondition} being satisfied with $\zeta=e_{n+1}$.
Therefore when $x_i < 0$, $\IP{\nu^{M_t}}{K_i} > 0$, when $x_i > 0$ we have $\IP{\nu^{M_t}}{K_i} < 0$, and when $x_i = 0$ we have $\IP{\nu^{M_t}}{K_i} = 0$.
This is precisely \eqref{initialsphere}.
\end{proof}

We now collect some additional results needed for the proof of our main theorem.
The first is the evolution of the quantities $s_i$.

\begin{prop}
Let $M_t = F(M^n,t)$ be a mean curvature flow of hypersurfaces in $\R^{n+1}$.
The quantities $s_i=\IP{\nu^{M_t}}{K_i}$, $i=1,\ldots,n$ satisfy the evolution equations
\begin{align}
\Big(\frac{d}{dt}-{\Delta}_{M_t}\Big)s_i\ =\ |A^{M_t}|^2  s_i,
\end{align}
where we denoted by $A^{M_t}$ the second fundamental form of $M_t$.
\label{killingevolution}
\end{prop}
\begin{proof}
First we compute
\begin{align*}
\frac{d}{dt}s_k\ & =\ \frac{d}{dt}\IP{\nu^{M_t}}{K_k\circ F_t}\ =\
\IP{\nabla H}{K_k\circ F_t}\ -\ H\IP{\nu^{M_t}}{(\D_{\nu^{M_t}}K_k)(F_t)}\\
& =\ \IP{\nabla H}{K_k\circ F_t},
\end{align*}
where we have used the antisymmetry of Killing vector fields implying that $\IP{{\D}_{V}K_k}{V}=0$ for every vector field $V$.
Let $\{{\tau}_i\}_{i=1,\ldots,n}$ be an orthonormal basis of $TM_t$.
In the calculations below we omit the composition of $K_i$ with $F_t$.
We continue by computing the Laplace-Beltrami operator applied to $s_k$:
\begin{align*}
{\nabla}_{{\tau}_i}s_k
\ =\ &\nabla_{{\tau}_i}\IP{\nu^{M_t}}{K_k}=\sum_{p=1}^{n}h_{ip}\IP{{\tau}_p}{K_k}+\IP{\nu^{M_t}}{{\D}_{{\tau}_i}K_k},\\
D_{{\tau}_j}{\nabla}_{{\tau}_i}s_k
\ =\ &\sum_{p=1}^{n}{\nabla}_{{\tau}_j}h_{ip}\IP{{\tau}_p}{K_k}+\sum_{p=1}^{n}h_{ip}\IP{{\D}_{{\tau}_j}{\tau}_p}{K_k} + \sum_{p=1}^{n}h_{ip}\IP{{\tau}_p}{{\D}_{{\tau}_j}K_k}\\
&+\ \sum_{p=1}^{n}h_{jp}\IP{{\tau}_p}{{\D}_{{\tau}_i}K_k}+\IP{\nu^{M_t}}{{\D}^2_{{\tau}_i,{\tau}_j}K_k}+\IP{\nu^{M_t}}{{\D}_{{\D}_{{\tau}_j}{\tau}_i}K_k},
\end{align*}
where we used the Weingarten equation
 and denoted by $h_{ij}$ the components of the second fundamental form $A^{M_t}$. We also compute
\[
{\D}_{{\tau}_i}{\tau}_j=-h_{ij}\nu^{M_t}+\displaystyle
\sum_{k=1}^{n}\Gamma_{ij}^{k}{\tau}_k,
\]
using again the definition of the second fundamental form $A^{M_t}=\big(h_{ij}\big)_{1\leq i,j\leq n}$ and Christoffel symbols.
For ease of computation we choose an orthonormal basis of the tangent space such that the Christoffel symbols vanish at the point where the computation is evaluated, that is $\Gamma_{ij}^{k}= 0$ for all
$i,j,k=1,\ldots,n$.
The local linearity of a Killing vector field causes the second derivative of $K_k$ to also vanish.
These considerations simplify the computation to:
\begin{align*}
{\Delta}_{M_t} s_k\ =\ &\sum_{i=1}^{n}
\IP{{\tau}_i}{{\D}_{{\tau}_i}\nabla
s_k}=\sum_{i=1}^{n}\sum_{p=1}^{n}{\nabla}_{{\tau}_i}h_{ip}\IP{{\tau}_p}{K_k}
-\sum_{i=1}^{n}\sum_{p=1}^{n}h_{ip}h_{ip}\IP{\nu^{M_t}}{K_k}\\
&+\
\sum_{i=1}^{n}\sum_{p=1}^{n}h_{ip}\IP{{\tau}_p}{{\D}_{{\tau}_i}K_k}+\sum_{i=1}^{n}\sum_{p=1}^{n}h_{ip}\IP{{\tau}_p}{{\D}_{{\tau}_i}K_k}
-\sum_{i=1}^{n}h_{ii}\IP{\nu^{M_t}}{{\D}_{\nu^{M_t}}K_k}.
\end{align*}
Using
the Codazzi equation on the first term we obtain
\begin{align*}
\sum_{i=1}^{n}\sum_{p=1}^{n}{\nabla}_{{\tau}_i}h_{ip}\IP{{\tau}_p}{K_k}\
=\ \IP{\nabla H}{K_k}.
\end{align*}
The antisymmetry of Killing vector fields implies that $\IP{{\D}_{V}K_k}{V}=0$ for every vector field $V$.
This makes the last term in the computation of the Laplace-Beltrami operator applied to $s_i$ vanish.
To use this property on the rest of the terms we consider local normal coordinates which diagonalise the second fundamental form as in \cite{andrews1994contraction}.
This eliminates all the first order terms containing $\D K_k$, leaving us with the following expression:
\begin{align*}
\Delta s_k\ =\ \IP{\nabla
H}{K_k}-\sum_{i=1}^{n}h_{ii}^2\IP{\nu^{M_t}}{K_k}=\IP{\nabla
H}{K_k}-|A_{M_t}|^2 s_k.
\end{align*}
If we put this last result together with the time derivative
computed above we finally obtain the desired evolution for $s_k$.
\end{proof}

We now employ the following result from Stahl \cite{thesisstahl}.
The problem treated in \cite{thesisstahl} is the mean curvature flow of immersions with a ninety-degree contact angle on a fixed hypersurface in ${\R}^{n+1}$, but here we are only interested in using the setting of $\Sigma$ as the unit sphere in $\R^{n+1}$.

\begin{prop}[Stahl \cite{thesisstahl}]
\label{sigmacurvaturerelated}
Let $M_t = F(M^n,t)$ be a mean curvature flow of hypersurfaces in $\R^{n+1}$ satisfying \eqref{immersion}.
Let $X\in \Sigma \cap M_t$, $v\in T_XM_t$ and $w:=v-\IP{v}{{\nu}^{\Sigma}}{\nu}^{\Sigma}\in T_X(M_t\cap \Sigma)$ be the projection of $v$ onto $T_X\Sigma$.
Then:
\begin{align*}
A^{M_t}(w,{\nu}^{\Sigma})&\ =\ -A^{\Sigma}(w,\nu^{M_t}),\\
A^{M_t}(v,{\nu}^{\Sigma})&\ =\ -A^{\Sigma}(w,\nu^{M_t})\ +\
\IP{v}{{\nu}^{\Sigma}}A^{M_t}({\nu}^{\Sigma},{\nu}^{\Sigma}),\\
\nabla_{\nu^{\Sigma}}H\ &=\ H A^{\Sigma}(\nu^{M_t},\nu^{M_t}).
\end{align*}
\end{prop}
So far the results have been obtained independently of the reflective symmetry. For the following we make use of the reflective symmetry to restrict the problem in a subcone of the ambient space.

Conditions \eqref{initialsphere} and the reflective symmetry of the evolving hypersurfaces allow us to restrict the problem to the cone
\[
\Q^+ = \{(x_1,\ldots,x_n)\,:\,x_i\geq 0 \ \text{for each} \ i=1,\ldots,n\}.
\]
Let us define ${M_t}^{+}:=M_t\cap \Q^+$. From the initial condition \eqref{initialsphere} we have that on ${M_0}^{+}$, $s_i(X_0)<0$, with zero boundary values on ${M_0}^{x_i=0}=M_0 \cap \{x_i=0\}$.
There are $n+1$ more boundaries of the domain: the free boundary at the intersection with the sphere $\Sigma$ which we denote by ${\partial}_N{M_0}^{+}$, the fixed Dirichlet boundary on the fixed radius outside the unit
sphere, which we denote by ${\partial}_D{M_0}^{+}$, and ${M_0}^{x_j=0}=M_0 \cap \{x_j=0\}$ for all $j\neq i$.

The following result states that tilt points do not occur on planes of reflection.

\begin{prop}
\label{propepsilonboundariesn}
Let $M_t = F(M^n,t)$ be an RGMCF with $|\IP{F_0}{e_{n+1}}|\le1$.
For every $j\in\{1,\ldots,n\}$ and every $X=F_t(p)\in \{x_j=0\}$, $p\in M^n$, there exists $i\in\{1,\ldots,n\}$, $i\neq j$ such that
\[
s_i(X) \ne 0.
\]
\end{prop}
\begin{proof}
Suppose that there exists a $j\in\{1,\ldots,n\}$ such that there exists a point $X=F_t(p)\in \{x_j=0\}$, $p\in M^n$, for which we have
\[
s_i(X) = 0,\quad\text{for all }\quad i\in\{1,\ldots,n\}.
\]
Notice that by smoothness of $s_j$ we automatically have $s_j=0$ on $\{x_j=0\}$.
At $X$, a boundary point for the subcone $\Q^+$, consider an orthonormal basis $\{{\tau}_1,\ldots,{\tau}_n\}$ of $T_X M_t$ such that
\begin{align*}
{\tau}_i|_X\ &\in T(M_t \cap \{x_j=0\}),\ \ \text{and}\\
{\tau}_n|_X(x_1,\ldots,x_{n+1}) \ &=\ (x_j,0,\ldots,0,-x_1,0,\ldots,0)|_X\\
&=\ {\nu}_{M_t\cap\{x_j=0\}}|_X \in T^{\perp}(M_t\cap \{x_j=0\}) \cap TM_t.
\end{align*}
Here we have used the fact that we can always find an $x_k\neq 0$ since, by the comparison principle, the origin can never be one of the points discussed here.
Without loss of generality we can chose this to be $x_1$ which explains our choice of $\tau_n$.

At this particular point $X\in\{x_j=0\}$ the unit normal of the evolving hypersurfaces $M_t$ is of the form
\begin{align*}
\nu^{M_t}|_X=\frac{1}{\sqrt{\sum_{i=1}^{n+1}x_i^2}}(x_1,\ldots,x_{n+1}).
\end{align*}
The above can be proved by induction using the fact that $s_i=0$ for all $i\in\{1,\ldots,n\}$.
We omit the details here.

The specific form of the normal is enough to exclude the presence of this type of point $X$ at the corner with the Neumann boundary where $\IP{\nu^{M_t}}{\nu^\Sigma}=0$ and $\nu^\Sigma=-(x_1,\ldots,x_{n+1})$ for the
sphere $\Sigma$. If the point $X$ would be at the corner with the Dirichlet boundary where $\{x_{n+1}=0\}$ then the specific form of the normal would imply the graph is vertical.
This possibility is excluded by a standard barrier argument, which concludes that the gradient of the graph function is bounded.
This idea of barriers on the Dirichlet boundary will be exploited in detail in the proof of Proposition \ref{propKilling}.

Note that for $\Q^+\cap M_t$ we have $s_i<0$ for all $i\in \{1,\ldots,n\}$ and $X$ would be a first boundary point where a new maximum with value zero would be attained for all the $s_i$ quantities. We will now employ
the maximum principle, using the parabolic evolution equation for the $s_i$ quantities and the Hopf Lemma to show that such a point $X$ can not exist.
First from the Hopf Lemma we see that at $X$ for all $i\in\{1,\ldots,n\}$ with $i\neq j$
\begin{align}
0&<\nabla_{\tau_n} s_i = \sum_{s=1}^n A^{M_t}(\tau_n,\tau_s)\IP{\tau_s}{K_i}+\IP{\D_{\tau_n}K_i}{\nu^{M_t}}\notag\\
&= \sum_{s=1}^{n-1} A^{M_t}(\tau_n,\tau_s)\IP{\tau_s}{K_i},\label{epsilon1}
\end{align}
where we have used that $\D_{\tau_n}K_i \equiv 0$ and $\IP{\tau_n}{K_i}=0$ when $x_j=0$ for $i \neq j$.
The same Hopf Lemma will apply for $i=j$ but different terms are non-vanishing due to our choice of tangent vectors.
At $X$ we have again
\begin{align}
0&<\nabla_{\tau_n} s_j = \sum_{s=1}^n A^{M_t}(\tau_n,\tau_s)\IP{\tau_s}{K_j}+\IP{\D_{\tau_n}K_j}{\nu^{M_t}}\notag\\
&=  A^{M_t}(\tau_n,\tau_n)\IP{\tau_n}{K_j} + \frac{1}{\sqrt{\sum_{s=1}^{n+1}x_s^2}}x_1x_{n+1}\notag\\
&=  -A^{M_t}(\tau_n,\tau_n)x_1x_{n+1} + \frac{1}{\sqrt{\sum_{s=1}^{n+1}x_s^2}}x_1x_{n+1}\notag\\
&=  x_1x_{n+1}\left(-A^{M_t}(\tau_n,\tau_n) + \frac{1}{\sqrt{\sum_{s=1}^{n+1}x_s^2}}\right),\label{epsilon2}
\end{align}
where we have used $\IP{\tau_s}{K_j}=0$ for all $s\in\{1,\ldots,n-1\}$ when $x_j=0$ and since $\tau_s\in\{x_j=0\}$.
But also at $X$ found on $\{x_j=0\}$ we have $s_j=0$ so we can differentiate in directions tangent to the boundary contained in $\{x_j=0\}$ to obtain for all $s\in\{1,\ldots,n-1\}$
\begin{align}
0&=\nabla_{\tau_s}s_j=\nabla_{\tau_s}\IP{\nu^{M_t}}{K_j}=\sum_{p=1}^{n}A^{M_t}(\tau_p,\tau_s)\IP{\tau_p}{K_j}+\IP{\D_{\tau_s}K_j}{\nu^{M_t}}\notag\\
&= - A^{M_t}(\tau_s,\tau_n)x_1x_{n+1},\label{epsilon3}
\end{align}
where we have used $\nu_j=0$ and $\tau_s\in\{x_j=0\}$ to show that $\IP{\D_{\tau_s}K_j}{\nu^{M_t}}=0$, also $\IP{\tau_p}{K_j}=0$ for all $p\neq n$ and $\IP{\tau_n}{K_j}=-x_1x_{n+1}$. Since \eqref{epsilon2} implies
$x_1x_{n+1}\neq 0$, \eqref{epsilon3} shows that $A^{M_t}(\tau_s,\tau_n)=0$ for all $s\in\{1,\ldots,n-1\}$ and this contradicts \eqref{epsilon1}. So our assumption of the existence of such a point $X$ was false and the
proposition is proved.
\end{proof}

\begin{rmk}
The initial height bound used above can be improved through use of catenoid comparison as in Section 2.
\end{rmk}

We can make Proposition \ref{propepsilonboundariesn} above even stronger by showing that on any reflection hyperplane $\{x_i=0\}$ the only $s_j$ quantity which may vanish is the one which must vanish, that is, the
particular $s_i$ quantity corresponding to that hyperplane.

\begin{prop}
Let $M_t = F(M^n,t)$ be an RGMCF with $|\IP{F_0}{e_{n+1}}|\le1$.
For all $j\in\{1,\ldots,n\}$ and all points $X=F_t(p)\in \{x_j=0\}$ with $p\in M^n$ we have $s_k(X)\neq 0$ for all $k\in\{1,\ldots,n\}$, $k\neq j$.
\label{propepsilonboundaries}
\end{prop}
\begin{proof}
As mentioned above, conditions \eqref{initialsphere} and the reflective symmetry of the evolving hypersurfaces allows us to restrict the problem to the cone $\Q^+$.
Suppose there exists a $j\in\{1,\ldots,n\}$ such that for some point $X=F_t(p)\in \{x_j=0\}$, $p\in M^n$, there exists $k\in\{1,\ldots,n\}$, $k\neq j$, where we have
$s_k(X)=0$. Once again we have that $X$ is a boundary point for the subcone $\Q^+$.
Note that the smoothness of $s_j$ already implies $s_j=0$ on $\{x_j=0\}$.
Consider an orthonormal basis $\{{\tau}_1,\ldots,{\tau}_n\}$ of $T_X M_t$ such that
\begin{align*}
{\tau}_i|_X\ &\in T(M_t \cap \{x_j=0\}),\ \ \text{and}\\
{\tau}_n|_X(x_1,\ldots,x_{n+1}) \ &=\ (x_j,0,\ldots,0,-x_1,0,\ldots,0)|_X\\
&=\ {\nu}_{M_t\cap\{x_j=0\}}|_X \in T^{\perp}(M_t\cap \{x_j=0\}) \cap TM_t.
\end{align*}
Here we have used the fact that we can always find an $x_s\neq 0$ since the origin can never be one of the points discussed here.
Without loss of generality we can chose this to be $x_1$ which explains our choice of $\tau_n$. The arguments stand even if $k=1$.

At this particular point $X\in\{x_j=0\}$, due to $s_j=s_k=0$, the unit normal of the evolving hypersurfaces $M_t$ satisfies
\begin{align}
\nu_j=0\qquad \text{ and }\qquad x_{n+1}\nu_k=x_k\nu_{n+1}.\label{normalrelations}
\end{align}

If the point $X$ would be at the corner with the Dirichlet boundary where $\{x_{n+1}=0\}$ we use Lemma \ref{propInitialD} to see that $s_k=0$ only on the $\{x_k=0\}$ hyperplane, which would imply that the vector field
$K_k$ vanishes.
This contradicts relation \eqref{epsilon21} obtained through application of the Hopf Lemma below, which applies also in this case. The existence of a parabolic frustum, \cite{lieberman1996second} in which we can apply Hopf lemma at such a corner point is given by the nature of the Dirichlet boundary and the reflective symmetry.

Suppose that $X$ does not lie on a corner formed by the Neumann and Dirichlet boundaries.
Note that for $\Q^+\cap M_t$ we have $s_k<0$ and $X$ would be a first boundary point where a zero maximum would be attained.
We will now employ the maximum principle, using the parabolic evolution equation for $s_k$ and the Hopf Lemma to show that such a point $X$ can not exist.
First from the Hopf Lemma we see that at $X$ we have
\begin{align}
0&<\nabla_{\tau_n} s_k = \sum_{s=1}^n A^{M_t}(\tau_n,\tau_s)\IP{\tau_s}{K_k}+\IP{\D_{\tau_n}K_k}{\nu^{M_t}}\notag\\
&= \sum_{s=1}^{n-1} A^{M_t}(\tau_n,\tau_s)\IP{\tau_s}{K_k},\label{epsilon21}
\end{align}
where we have used that $\D_{\tau_n}K_k \equiv 0$ and $\IP{\tau_n}{K_k}=0$ when $x_j=0$ for $k \neq j$.
The same Hopf Lemma will apply for $j$ but different terms are non vanishing due to our choice of tangent vectors.
At $X$ we have again
\begin{align}
0&<\nabla_{\tau_n} s_j = \sum_{s=1}^n A^{M_t}(\tau_n,\tau_s)\IP{\tau_s}{K_j}+\IP{\D_{\tau_n}K_j}{\nu^{M_t}}\notag\\
&=  A^{M_t}(\tau_n,\tau_n)\IP{\tau_n}{K_j} + x_1\nu_{n+1}\notag\\
&=  x_1\big(-A^{M_t}(\tau_n,\tau_n)x_{n+1} +\nu_{n+1}),\label{epsilon22}
\end{align}
where we have used $\IP{\tau_s}{K_j}=0$ for all $s\in\{1,\ldots,n-1\}$ when $x_j=0$ since $\tau_s\in\{x_j=0\}$.
But also at $X$ found on $\{x_j=0\}$ we have $s_j=0$ so we can differentiate in directions tangent to the boundary contained in $\{x_j=0\}$ to obtain for all $s\in\{1,\ldots,n-1\}$
\begin{align}
0&=\nabla_{\tau_s}s_j=\nabla_{\tau_s}\IP{\nu^{M_t}}{K_j}=\sum_{p=1}^{n}A^{M_t}(\tau_p,\tau_s)\IP{\tau_p}{K_j}+\IP{\D_{\tau_s}K_j}{\nu^{M_t}}\notag\\
&= - A^{M_t}(\tau_s,\tau_n)x_1x_{n+1},\label{epsilon23}
\end{align}
where we have used $\nu_j=0$ (from the reflective symmetry) and $\tau_s\in\{x_j=0\}$ to show that $\IP{\D_{\tau_s}K_j}{\nu^{M_t}}=0$, also $\IP{\tau_p}{K_j}=0$ for all $p\neq n$ and $\IP{\tau_n}{K_j}=-x_1x_{n+1}$.
Since \eqref{epsilon22} implies $x_1\neq 0$, \eqref{epsilon23} shows that $A^{M_t}(\tau_s,\tau_n)x_{n+1}=0$ for all $s\in\{1,\ldots,n-1\}$.
If $x_{n+1}=0$ then \eqref{normalrelations} implies that either $x_k=0$ or $\nu_{n+1}=0$.
The latter, $\nu_{n+1}=x_{n+1}=0$, contradicts the strict sign of \eqref{epsilon22}.
If $x_k=0$ then $K_k\equiv 0$ since $x_{n+1}=0$ too.
This shows that $\IP{\tau_s}{K_k}=0$ and contradicts \eqref{epsilon21}.
So all that remains is $A^{M_t}(\tau_s,\tau_n)=0$ for all $s\in\{1,\ldots,n-1\}$ which again contradicts \eqref{epsilon21}.

If the point $X$ is found on the corner with the Neumann boundary then the same arguments apply. 
If $x_{n+1}=0$ and $\nu_{n+1}=0$ we again obtain a contradiction with the strict sign of \eqref{epsilon22}. The existence of a parabolic frustum, \cite{lieberman1996second} in which we can apply Hopf lemma at such a corner point is given by the ninety degree boundary contact condition on the sphere, which provides enough space for the frustum to exist.
This completes our proof and shows that the existence of such a point $X$ was false.
\end{proof}

At a point on the Neumann boundary where $s_i=0$ for some $i=1,\ldots,n$, the components of the second fundamental form satisfy certain relations which we now describe.

\begin{prop}
Let $M_t = F(M^n,t)$ be an RGMCF with $|\IP{F_0}{e_{n+1}}|\le1$ restricted in the positive subcone $\Q^+$.
Consider a point on the Neumann boundary $X=F_t(p)\in{\partial}_N M_t \subset \Sigma$ for some $p\in {\partial}_N M^n$ where for the first time we have
\[
s_i(X)=0
\]
for some $i=1,\ldots,n$.
Then, for an orthonormal basis $\{{\tau}_1,\ldots,{\tau}_n\}$ of $T_X M_t$
such that
\begin{align*}
{\tau}_i|_X\ \in T_X\partial_N M_t\ \ \text{and}\ \ {\tau}_n|_X\ =\
{\nu}^{\Sigma}|_X\ =\ {\nu}_{{\partial}_N M_t}|_X,
\end{align*}
we have
\begin{align*}
\sum_{s=1}^{n-1}\IP{\tau_s}{K_i}A^{M_t}\big|_X({\tau}_s,{\nu}^{\Sigma})\ >\ 0.
\end{align*}
\label{lemtiltsphere}
\end{prop}
\begin{proof}
From the conditions imposed on $s_i$ at and around the point $X$, $s_i$ has attained a boundary maximum at this point, after being negative everywhere in the interior.
Proposition \ref{killingevolution} shows that $s_i$ satisfies a parabolic evolution equation, allowing us to apply the Hopf Lemma at the point $X$:
\begin{align*}
0\ <\
{\nabla}_{{\tau}_n}s_i={\nabla}_{{\tau}_n}\IP{\nu^{M_t}}{K_i}=\sum_{i=1}^nA^{M_t}({\tau}_s,{\tau}_n)\IP{{\tau}_s}{K_i}
+\IP{{\D}_{{\tau}_n}K_i}{\nu^{M_t}},
\end{align*}
where we have used the Gauss-Weingarten equations to express
derivatives of the normal in tangential directions. Now we know
that at $X$ we have ${\tau}_n={\nu}^{\Sigma}=-{{\nu}_s}/|{\nu}_s|$, where ${\nu}_s(x_1,\ldots,x_{n+1})=(x_1,\ldots,x_{n+1})$ is the
position vector in ${\R}^{n+1}$ and it is always normal to the sphere
$\Sigma$. Then
\begin{align*}
\IP{{\D}_{{\tau}_n}K_i}{\nu^{M_t}}|_X\ =-\
\frac{1}{|{\nu}_s|}\IP{{\D}_{{\nu}_s}K_i}{\nu^{M_t}}\ =-\
\frac{1}{|{\nu}_s|}\IP{K_i}{\nu^{M_t}}\ =\ 0,
\end{align*}
since $s_i=\IP{\nu^{M_t}}{K_i}=0$ at $X$. Also we have  $\IP{{\tau}_n}{K_i}\ = \IP{\nu^{\Sigma}}{K_i}=\ 0$. We have thus shown that
\begin{align*}
0\ < \sum_{i=1}^{n-1}\ A^{M_t}|_X({\tau}_s,{\nu}^{\Sigma})\IP{\tau_s}{K_i},
\end{align*}
which gives us the desired result.
\end{proof}

The next result shows that we can preserve the sign of $s_i$ with the use of the extra conditions \eqref{initialsphere}.

\begin{prop}
Let $M_t = F(M^n,t)$ be an RGMCF with $|\IP{F_0}{e_{n+1}}|\le1$.
The flow preserves conditions \eqref{initialsphere} for all time.
\label{propKilling}
\end{prop}
\begin{proof}
The proof is based on the application of the maximum principle for $s_i$ on ${M_t}^{+}:=M_t\cap \Q^+$. For the convenience of the reader we remind here our sign convention on the subcone and the definition of our $n+2$ boundaries to $M_t^{+}$.
From the initial condition \eqref{initialsphere} we have that on ${M_0}^{+}$, $s_i(X_0)<0$, with zero boundary values on the boundary ${M_0}^{x_i=0}=M_0 \cap \{x_i=0\}$.
There are $n+1$ more boundaries of the domain: the free boundary at the intersection with the sphere $\Sigma$ which we denote by ${\partial}_N{M_0}^{+}$, the fixed Dirichlet boundary on the fixed radius outside the unit
sphere, which we denote by ${\partial}_D{M_0}^{+}$, and ${M_0}^{x_j=0}=M_0 \cap \{x_j=0\}$ for all $j\neq i$.

From Proposition \ref{killingevolution} and the maximum principle on ${M_t}^{+}$ we know that the sign of $s_i$ can be preserved for all times, if on the boundaries we do not get any `new' zero values (which also are
maximal values of $s_i$ on $\overline{{M_t}^{+}}$).
On the $n-1$ boundaries which come from the reflective symmetry we can not have a new $0$ value as shown in Proposition \ref{propepsilonboundaries}.
So we turn our attention to the two boundaries which can make a difference and change the sign of $s_i$.

First we need to exclude the possibility that $s_i$ might take a zero value on the Neumann boundary.
Suppose that there is a point $X=F_t(p)$ on ${\partial}_N{M_t}^{+}\subset {\Sigma}$ for some $p\in {\partial}_N M^n$ where we have for the first time in the evolution of the graph that $s_i(X)=0$.
At this point of the boundary we consider an orthonormal basis $\{{\tau}_1,\ldots,{\tau}_n\}$ of $T_X M_t$, chosen such that we have at $X$
\begin{align*}
{\tau}_i|_X\ \in T_X\partial_N M_t^{+}\ \ \text{and}\ \ {\tau}_n\ =\
{\nu}^{\Sigma}={\nu}_{{\partial}_N {M_t}^{+}}.
\end{align*}
Now using the result of Proposition \ref{lemtiltsphere} we see that at $X$
\begin{align}
0\ < \sum_{i=1}^{n-1}\ A^{M_t}|_X({\tau}_s,{\nu}^{\Sigma})\IP{\tau_s}{K_i}, \label{valueattilt}
\end{align}
where $A^{M_t}$ is the second fundamental form.
Using a result of Stahl \cite{thesisstahl}, which we quoted in Proposition \ref{sigmacurvaturerelated}, we know that
\[
A^{M_t}({\tau}_s,{\nu}^{\Sigma})=-A^{\Sigma}({\tau}_s,\nu^{M_t}).
\]
for all $s\in\{1,\ldots,n-1\}$.
This is helpful since at a boundary point the tangent space of $\Sigma$ is spanned by $\{{\tau}_1,\ldots,\tau_{n-1}, \nu^{M_t}\}$.
Since $\Sigma$ is a sphere and the basis is an orthogonal one, the directions defined by its vectors are the principal directions at the point $X$.
Thus the second fundamental form of $\Sigma$ is diagonal at $X$.
Using the relation between the off-diagonal elements of the second fundamental form of $M_t$ and $\Sigma$ we can see that
\[
A^{M_t}({\tau}_s,{\nu}^{\Sigma})=-A^{\Sigma}({\tau}_s,\nu^{M_t})=0
\]
for all $s\in\{1,\ldots,n-1\}$, which contradicts \eqref{valueattilt}.
Therefore there does not exist a point on the Neumann boundary where $s_i$ changes sign.

Now the other problem is if the $s_i$ quantity changes sign on the Dirichlet boundary.
This cannot be the case since we started with an initial graph in the $e_{n+1}$ direction.
The standard construction of barriers on Dirichlet boundaries shows that this relation is preserved for all times of existence.
Finally, using Lemma \ref{propInitialD} we see that on the Dirichlet boundary relation \eqref{graphcondition} for $\zeta=e_{n+1}$ is equivalent to $s_i$ being negative.

Using the reflective symmetry we complete the proof of conditions \eqref{initialsphere}.
\end{proof}

\begin{rmk}
The condition imposed on the initial height, that $|\IP{F_0}{e_{n+1}}|\leq 1$, is there to prevent the graphs from flowing to the North or South Pole of the sphere $\Sigma$, points in which the vector field $\xi$ is not
defined.
The height bound can be preserved in at least two ways.

One of them is by constructing radially symmetric barriers which are above and below the maximal height of the initial graph. Since the radially symmetric solutions have a height bound from the results of the previous
section, our general reflective symmetric graph also enjoys a height bound.

The second method is to use the same argumentation as \cite[Chapter 6]{thesisvulcanov} developed for general graphs.
The Neumann boundary condition and the convention that we take the unit normal to the sphere $\Sigma$ to be pointing away from the evolving surfaces implies $\IP{{\nu}^{\Sigma}}{e_{n+1}}\leq 0$ above the ${\R}^n$ plane
and the opposite sign below.
Using this one can prove that the height of the graphs remains bounded for all times by the initial bound.
Using the result of Theorem \ref{thmlongtimeexistencecatenoid} the initial height can be taken up to and including the maximal height of the sphere.
\end{rmk}

Perhaps a little surprisingly, one can show that while the gradient is bounded the mean curvature satisfies a \emph{uniform} bound.

\begin{prop}
Let $M_t = F(M^n,t)$ be an RGMCF with $|\IP{F_0}{e_{n+1}}|\le1$.
There exists an absolute constant $C<\infty$ such that
\begin{align*}
\sup_{M_t}|H| \leq C \sup_{M_0} |H|,
\end{align*}
for all $t\in[0,T)$.
\label{propCurvaturebound}
\end{prop}
\begin{proof}
The proof is based once again on the use of the maximum principle and the Hopf Lemma.
In the following we modify an idea of Ecker and Huisken \cite{ecker1989mce} allowing one to obtain a uniform curvature bound once a gradient bound is in-hand.
Proposition \ref{propKilling} gives us that the quantities $s_i$ preserve the strict negative sign on the quadrant ${M_t}^{+}$, which is equivalent to a gradient bound.
Also from Proposition \ref{propepsilonboundaries} we know that on any of the plane of symmetry the quantity $\sum_{i=1}^{n} s_i$ is strictly negative,that is non vanishing.

Consider the quantity $X\mapsto\frac{H^2}{(\sum_{i=1}^{n} s_i)^2}(X):M_t^{+}\rightarrow \R$.
Using the reflective symmetry we see that it is enough to work on $M_t^{+}$.
After the same computation as in \cite{ecker1989mce} and using the evolution of the mean curvature found in \cite{huisken1986cch} we find that $\frac{H^2}{(\sum_{i=1}^{n} s_i)^2}$ satisfies the parabolic evolution equation
\begin{align*}
\big(\frac{d}{dt}-{\Delta}^{M_t}\big)\ \frac{H^2}{(\sum_{i=1}^{n} s_i)^2} \ \leq \
2\frac{\nabla (\sum_{i=1}^{n} s_i)}{\sum_{i=1}^{n} s_i}\cdot \nabla{\frac{H^2}{(\sum_{i=1}^{n} s_i)^2}}.
\end{align*}
From the above evolution and the use of the maximum principle with the bounded vector field $a=\frac{\nabla (\sum_{i=1}^{n} s_i)}{\sum_{i=1}^{n} s_i}$, we see that as long as we exclude maxima of the above quantity on the boundaries we obtain the result.

The Dirichlet boundary ${\partial}_D M_t$ is a non-issue, since the compatibility condition $H|_{{\partial}_D M_0}\equiv0$ is preserved for all times (see \cite{thesisvulcanov} for more details on this).

On the Neumann boundary ${\partial}_N M_t$ we apply a Hopf Lemma argument.
Assume that there is a point $X=F(p,t)\in{\partial}_NM_t$ such that $\frac{H^2}{(\sum_{i=1}^{n} s_i)^2}$ attains a maximum at $X$.
At this point choose an orthonormal basis $\{{\tau}_1,\ldots,{\tau}_{n}\}$ of the tangent space $T_XM_t$ such that ${\tau}_i\in T{\partial}_N M_t$ for all $i\in\{1,\ldots,n-1\}$ and ${\tau}_n={\nu}^{\Sigma}$ at $X$.
Then the Hopf Lemma implies
\begin{align}
0\ <\ {\nabla}_{{\nu}^{\Sigma}}\frac{H^2}{(\sum_{i=1}^{n} s_i)^2}\ =\
2\frac{H}{(\sum_{i=1}^{n} s_i)^2}{\nabla}_{{\nu}^{\Sigma}}H\ -\
2\frac{H^2}{(\sum_{i=1}^{n} s_i)^3}{\nabla}_{{\nu}^{\Sigma}}(\sum_{i=1}^{n} s_i).
\label{hopfforcurvature}
\end{align}
Using Proposition \ref{sigmacurvaturerelated} we replace in the
first term
\begin{align*}
{\nabla}_{{\nu}^{\Sigma}}H\ =\ H
A^{\Sigma}(\nu^{M_t},\nu^{M_t})\ =\ -H,
\end{align*}
where we have used that $\Sigma$ is a sphere and that the unit
normal to $\Sigma$ points away from the evolving surfaces. We now
turn our attention to the second term in \eqref{hopfforcurvature},
and compute for any $i\in\{1,\ldots,n\}$:
\begin{align*}
{\nabla}_{{\nu}^{\Sigma}}s_i\ &=\
{\nabla}_{{\nu}^{\Sigma}}\IP{\nu^{M_t}}{K_i}\\
&=\ \sum_{s=1}^n A^{M_t}({\tau}_s,{\nu}^{\Sigma})\ \IP{{\tau}_s}{K_i}\ +\
\IP{\nu^{M_t}}{{\D}_{{\nu}^{\Sigma}}K_i}\\* &=\
\IP{\nu^{M_t}}{{\D}_{{\nu}^{\Sigma}}K_i},
\end{align*}
where we have used, as in the proof of Proposition
\ref{propKilling}, the relation
\begin{align*}
A^{M_t}({\tau}_s,{\nu}^{\Sigma})\ =\ -\
A^{\Sigma}({\tau}_s,\nu^{M_t})\ =\ 0,
\end{align*}
for all $s\in\{1,\ldots,n-1\}$ since ${\tau}_s\in T{\partial}_N M_t\subset T \Sigma$, ${\tau}_s$
is perpendicular to $\nu^{M_t}\in T \Sigma$, and $\Sigma$ is a
sphere. We have also used the fact that
$\IP{K_i}{{\nu}^{\Sigma}}=0$. Noting that
${\nu}^{\Sigma}=-{\nu}_s/|\nu_s|$, where we remind the reader that
${\nu}_s$ is the position vector, the last term in the above
computation simplifies to
\begin{align*}
{\nabla}_{{\nu}^{\Sigma}}s_i\ =\
\IP{\nu^{M_t}}{{\D}_{{\nu}^{\Sigma}}K_i}\ =\ -\
\IP{K_i}{\nu^{M_t}}\ =\ -s_i.
\end{align*}
Returning to \eqref{hopfforcurvature} we obtain a contradiction:
\begin{align*}
0\ < {\nabla}_{{\nu}^{\Sigma}}\frac{H^2}{(\sum_{i=1}^{n} s_i)^2}\ = \ -\
2\frac{H^2}{(\sum_{i=1}^{n} s_i)^2} \ +\ 2\frac{H^2}{(\sum_{i=1}^{n} s_i)^2}\ =\ 0.
\end{align*}
Therefore we do not have a maximum on the Neumann boundary for $\frac{H^2}{(\sum_{i=1}^{n} s_i)^2}$ at any positive time.

Due to the reflective symmetry on the $n$ boundaries given by $M_t \cap \{x_i=0\} $ for any $i\in\{1,\ldots,n\}$ we see that $\nu^{M_t} \in T\{x_i=0\}$, which tells us that the evolving mean curvature flow solution will be perpendicular on the hyperplanes of reflection. Therefore we have a mean curvature flow solution evolving with a ninety degree angle on a hyperplane. We can then use the results found in \cite{wheelermean}, Proposition 3.7 to exclude the appearance of maximum points on the $n$ boundaries given by the reflective hyperplanes.
Thus
\begin{align*}
\sup_{M_t}|H|\ \leq\
\frac{\sup_{M_t}|\sum_{i=1}^{n} s_i|}{\inf_{M_0}|\sum_{i=1}^{n} s_i|}\sup_{M_0}|H|.
\end{align*}
Noting that $\displaystyle \sup_{M_t}|\sum_{i=1}^{n} s_i|\leq \sup_{M_t}\sum_{i=1}^n|K_i|\leq
n\sup_{M_t}|{\nu}_s|\leq n\sup_{M_0}|{\nu}_s|$, and using the fact
that the height is bounded by the initial bound (see Lemma \ref{LMuniformheight})
gives us the existence of the global constant $C<\infty$ as
desired.
\end{proof}

We are now able to prove Theorem \ref{thmLongsphere}.

\begin{proof}[Proof of Theorem \ref{thmLongsphere}.]
As long as the immersion exists the non-tilting result from Proposition \ref{propKilling} can be applied to each of the quantities $s_i$.
This gives that relation \eqref{graphcondition} for $\zeta=\xi$ is preserved for all time.
We can therefore write our immersions as graphs in the $\xi$ direction for all time.
The sign preservation of \eqref{graphcondition} for $\zeta=e_{n+1}$ comes from the parabolic evolution that the quantity $\IP{\nu^{M_t}}{e_{n+1}}$ satisfies on the interior and the fact that the bad behaviour on the two
boundaries, Neumann and Dirichlet, for this quantity is equivalent to bad behaviour for the quantity $\IP{\nu^{M_t}}{\xi}$, which is prevented by the Proposition \ref{propKilling}.
\end{proof}

\begin{rmk}[Time dependent gradient bounds]
The sign preservation of the relation \eqref{graphcondition} provides us with a bound for the gradient of the associated scalar function.
By preserving for all times the positivity of the quantity $\IP{\nu^{M_t}}{e_{n+1}}$ we know that for all times of existence the surfaces can be written as a graph in the $\xi$ direction.
The bound is not uniform in time, hence for a long time existence result one would also require bounds on the full second fundamental form of the evolving surfaces.
The problem comes from the fact that the result of Proposition \ref{propKilling} is strongly dependent on the smoothness of the surface.

The usual proof of long time existence can take one of two paths.
One either provides bounds for all derivatives of the immersion for all times as done in \cite{ecker1989mce}, or refers to standard parabolic theory applied to the associated scalar evolution.
Bounding all the derivatives of the immersion requires information about these on the Neumann boundary, which at the moment we do not have.

In trying to apply the second approach we have encountered the following problem.
The associated scalar graph evolution for the problem \eqref{immersion} is quasilinear parabolic with an oblique derivative boundary condition on one of the boundaries and a Dirichlet condition on the other.
The long time existence theorems for these types of problems, as one can see from for example Corollary 8.10 and Theorem 8.3 in \cite{lieberman1996second}, require estimates on the $H_{1+\alpha}$ (for $\alpha\in(0,1)$)
norm independent of time.
Our gradient estimates are time dependent (in a non-obvious way), so obtaining $H_{1+\alpha}$ estimates from bounds on the height and gradient provides us with a time dependent bound, without any control on how the bound
grows in time.
To our knowledge this can be overcome if we know that for all times we have a hypersurface of class $C^2$.
Then, even at some finite final time we are able to apply the non-tilting arguments and obtain bounds on the gradient and then restart the flow.
\end{rmk}

\section{Mean concave (convex) graphs}
Let us first define precisely which setting we will be working in for this section.

\begin{defn}[GMCF${}^{H\le0}$]
We say that $M_t = F(M^n,t)$ is a graphical mean curvature flow with $H<0$ outside a sphere (\gmcfh) if
\begin{enumerate}
\item[(i)]   $M_t = F(M^n,t)$ is a one-parameter family of hypersurfaces evolving by mean curvature flow outside a standard sphere sphere in accordance with \eqref{immersion};
\item[(ii)]  $|h_0| < 1$;
\item[(iii)] $H(p,0)<0 $ for all $p\in \overline{M^n}\backslash \partial_DM^n$ (note that by the compatibility condition $H(p,t) = 0$ for all $(p,t)\in\partial_DM^n\times[0,T)$);
\item[(iv)]  the graphicality condition \eqref{graphcondition} holds for $\zeta = e_{n+1}$ and $t=0$.
\end{enumerate}
\end{defn}
The analogue of convex mean curvature flow of graphs can be defined by reversing the sign of the mean curvature assumptions in (\gmcfh). Below we treat the mean concave case, but all arguments carry through analogously in the case of mean convex initial data by a simple reflection.
Note that if $M_t = F(M^n,t)$ is a (\gmcfh) then it need only be \emph{initially} graphical and have \emph{initially} negative mean curvature.
That these properties are preserved follows from the work in this section.

Our goal is to prove Theorem \ref{thmSombrero}.
We shall establish the theorem by proving the following:
\vspace*{0.5cm}

\begin{tabular}{rp{3in}}
(Lemma \ref{LMconcavitypreserved})
 & Preservation of interior negativity of the mean curvature while the second fundamental form is bounded;
\\
(Lemma \ref{LMuniformheight})
 & Uniform height bounds;
\\
(Lemma \ref{LMuniformgradient})
 & Uniform gradient bounds for graphical solutions with non-positive mean curvature, two intrinsic dimensions, and initiall positive height; and
\\
(Lemma \ref{LMminimal})
 & Global in time uniformly bounded solutions converge to pieces of minimal surfaces.
\end{tabular}

\vspace*{0.5cm}
Given the above, the theorem then may then be proved as follows: Uniform bounds on $u$ and $Du$ follow by combining Lemma \ref{LMconcavitypreserved}, Lemma \ref{LMuniformheight}, and Lemma \ref{LMuniformgradient} on the
time interval $[0,T-\delta]$, where $T$ is the maximal time and $\delta>0$ is arbitrarily small.
The scalar evolution equation then becomes uniformly parabolic, and uniform estimates for all derivatives of the solution follow, in particular, uniform estimates for the second fundamental form.
This implies that the second fundamental form is uniformly bounded and we may conclude that Lemma \ref{LMconcavitypreserved} holds for all time.
We therefore conclude global existence.
Identification of the limit is a well-known standard argument (Lemma \ref{LMminimal}) which we have included here only for completeness.\\

We start with a preservation of the sign of the mean curvature. The extra perturbation term in the proof of the following Lemma is necessary to exclude the sensitive case of zeros propagating from the Dirichlet boundary into the interior without a maximum point at that zero.
\begin{lem}
\label{LMconcavitypreserved}
Suppose $M_t = F(M^n,t)$ is a \gmcfh.
Then for all $\delta > 0$ we have $H(p,t)<0$ for all $(p,t)\in (\overline{M^n}\backslash \partial_DM^n) \times [0,T-\delta]$.
\end{lem}
\begin{proof}
First note that on the time interval $[0,T-\delta]$ the second fundamental form is bounded uniformly.
Let us define $\lambda < \infty$ by setting
\[
\lambda = \sup_{(p,t)\in \overline{M^n}\times[0,T-\delta]} |A^{M_t}|^2(p,t)\,.
\]
In order to prove the lemma we consider the quantity $Q = He^{-\lambda t} - \varepsilon t$.
The evolution of $Q$ is given by
\[
(\partial_t - \Delta)Q = (|A^{M_t}|^2 - \lambda)He^{-\lambda t} - \varepsilon\,.
\]
Note that $Q(p,0) = H(p,0) < 0$ for $p\in M^n$.
Now on the Dirichlet boundary $Q = -\varepsilon t < 0$ for all $t>0$, and so $Q$ may not exceed its initial values on the Dirichlet boundary.
Furthermore, on $M^n$, if $Q$ exceeds its initial maximal value there must exist a new maximum for $Q$ at some point $(p_0,t_0)\in M^n\times[0,T-\delta]$ where $Q(p_0,t_0) = 0$ and so at this point
\begin{align*}
0 &\le (\partial_t - \Delta)Q(p_0,t_0)
\\
  &= (|A^{M_{t_0}}|^2 - \lambda)H(p_0,t_0)e^{- \lambda t_0} - \varepsilon
\\
  &= (|A^{M_{t_0}}|^2 - \lambda)\varepsilon t_0 - \varepsilon
\\
  &< 0\,,
\end{align*}
which is a contradiction.
Since $Q(p,0) = 0$ for $p\in \partial_DM^n$, it may have been possible that a zero could propagate instantaneously into the interior while on the Dirichlet boundary $Q$ is being dragged downward.
But this is also impossible, since it is not possible for $Q$ to attain a new interior positive maximum; indeed, suppose such a maximum occurs at $(p_1,t_1)$ where $Q(p_1,t_1) = \alpha = H(p_1,t_1)e^{-\lambda t_1} -
\varepsilon t_1$.
Rearranging, this implies
\[
H(p_1,t_1) = e^{\lambda t_1}(\alpha + \varepsilon t_1)
\]
and so, computing at the point $(p_1,t_1)$, we have
\begin{align*}
0 &\le (\partial_t - \Delta)Q(p_1,t_1)
\\
  &= (|A^{M_{t_1}}|^2 - \lambda)H(p_1,t_1)e^{- \lambda t_1} - \varepsilon
\\
  &= (|A^{M_{t_1}}|^2 - \lambda)(\alpha + \varepsilon t_1) - \varepsilon < 0\,.
\end{align*}

Finally, by Proposition \ref{sigmacurvaturerelated} and the Hopf Lemma, if a new maximum for $Q$ occurs at $(p_0,t_0)\in \partial_NM^n\times[0,T-\delta]$ satisfying $Q(p_0,t_0) = 0$ we must have
\[
0 < \nabla_{\nu^\Sigma}Q(p_0,t_0) = e^{-\lambda t}(H(p_0,t_0) A^\Sigma(\nu^{M_t}, \nu^{M_t})) = e^{-\lambda t} (- H(p_0,t_0)) = -\varepsilon t < 0\,,
\]
again a contradiction.

Therefore there can be no new maximum above the initial values for $Q$.
That is,
\begin{align*}
H(p,t)e^{-\lambda t} - \varepsilon t
 &= Q(p,t)
\\
 &\le
           \sup_{(p,t)\in \overline{M^n}\times\{0\}} Q(p,t)
\\
 &=         \sup_{(p,t)\in \overline{M^n}\times\{0\}} H(p,t)
\\
 &=         \sup_{p\in \overline{M^n}} H(p,0)
 = 0\,,
\end{align*}
and so, taking $\varepsilon\rightarrow0$, we conclude $H\le0$ on $\overline{M^n}\times[0,T-\delta]$ with $H(p,t)<0$ for all $(p,t)\in (\overline{M^n}\backslash \partial_DM^n) \times [0,T-\delta]$, as required.
\end{proof}

As in earlier sections, it is possible to obtain a priori height bounds through use of the comparison principle.
Below we show that there is also a direct argument based on the evolution equation of the height.

\begin{lem}
\label{LMuniformheight}
Suppose $M_t = F(M^n,t)$ is a mean curvature flow solution satisfying \eqref{immersion}.
Then for all $(p,t)\in\overline{M^n}\times[0,T)$ we have $|\IP{F}{e_{n+1}}| \leq C$, where $C$ depends only on $F_0$.
\end{lem}
\begin{proof}
Let us set $u(p,t) = \IP{F}{e_{n+1}}$. The evolution of $u$ is
\[
(\partial_t - \Delta)u = 0\,.
\]
In order to prove the lemma we consider the quantity $Q(p,t) = u^2(p,t)$.The evolution of $Q$ is
\[
(\partial_t - \Delta)Q = -2|\nabla u|^2 \,.
\]
We shall prove that $Q$ may not exceed its initial values. Based on its parabolic evolution the maximum principle tells us that $Q$ will be bounded by the maximum between the boundary values and its initial values.
On the Dirichlet boundary $u = h_0$ (recall the role that $h_0$ plays in \eqref{immersion}) and so $Q(p,t) = h_0^2 = Q(p,0)$ for $(p,t)\in\partial_DM^n\times(0,T)$.
If the new maximum occurs on the Neumann boundary by the Hopf Lemma this would imply for a choice of an orthonormal basis of the tangent space at that point as in Lemma \ref{propCurvaturebound}
\[
0 < 2u\IP{\nu^\Sigma}{e_{n+1}}\,.
\]
Since $\Sigma$ is a sphere, it is easy to see that for any $u$ we have $u\IP{\nu^\Sigma}{e_{n+1}} \le 0$, contradicting the above equation.

Therefore
\[
u^2(p,t) \le \max\{h_0^2,\sup_{p\in M^n}u^2(p,0)\} := \sqrt{C}
\]
where $C$ depends only on $F_0$ completing our proof.

\end{proof}

\begin{lem}
\label{LMuniformgradient}
Suppose $M_t = F(M^n,t)$ is a \gmcfh of surfaces, that is, $n=2$, with $\IP{F_0}{e_3} > 0$.
Then there exists an $s_0$ depending only on $F_0$ such that
\[
s(p,t) = \IP{\nu^{M_t}(p,t)}{e_3} \ge s_0
\]
for $(p,t) \in \overline{M^n}\times[0,T)$.
\end{lem}
\begin{proof}
We shall conduct as much of the proof as is possible in arbitrary dimension in order to highlight exactly where we require a restriction on the dimension of the solution.

In order to obtain a uniform gradient bound we must obtain a uniform positive lower bound for $s(p,t) = \IP{\nu^{M_t}(p,t)}{e_{n+1}}$.
The evolution of $s$ is
\[
(\partial_t-\Delta)s = |A^{M_t}|^2s\,.
\]
By the minimum principle for an initial positive $s$ we obtain that
\begin{align*}
\inf_{M_t} s\geq \min\{ \inf_{M_0} s, \inf_{\partial_D M_t} s, \inf_{\partial_N M_t} s\},
\end{align*}
for all $t\geq 0$.
Now on the Dirichlet boundary a standard barrier construction prevents the gradient from becoming unbounded and so we have $s(p,t) > s_0$ for $p\in\partial_DM^n$ and some $s_0 > 0$ depending on only on initial values.

It only remains to check the Neumann boundary. At an assumed point of minimum, we use again the Hopf Lemma, to obtain
\begin{equation*}
0 > \nabla_{\nu^\Sigma}s = \IP{\nabla_{\nu^\Sigma}\nu^{M_t}}{e_{n+1}}
\,.
\end{equation*}
Let us use local Fermi coordinates at the boundary to compute
\begin{align}
\IP{\nabla_{\nu^\Sigma}\nu^{M_t}}{e_{n+1}}
 &= \sum_{i=1}^{n-1} A^{M_t}(\tau_i,\nu^\Sigma)\IP{\tau_i}{e_{n+1}}
    + A^{M_t}(\nu^\Sigma,\nu^\Sigma)\IP{\nu^\Sigma}{e_{n+1}}
\notag\\
 &= H\IP{\nu^\Sigma}{e_{n+1}}
  - \sum_{i=1}^{n-1} A^{M_t}(\tau_i,\tau_i)\IP{\nu^\Sigma}{e_{n+1}}\,.
\label{EQhopfgradient}
\end{align}
Now since $s$ attains a new global minimum, this is also a new minimum on $\partial_NM^n$ and so at this point $\nabla_{\tau_i}s = 0$ for $i = 1,\ldots,n-1$.
That is,
\[
0 = \nabla_{\tau_i}\IP{\nu^{M_t}}{e_{n+1}}
  = \sum_{j=1}^{n-1}A^{M_t}(\tau_j,\tau_i)\IP{\tau_j}{e_{n+1}}
   + A^{M_t}(\tau_i, \nu^\Sigma)\IP{\nu^\Sigma}{e_{n+1}}
\]
Since $\Sigma$ is a sphere, we have from Proposition \ref{sigmacurvaturerelated} that $A^{M_t}(\tau_i, \nu^\Sigma) = 0$ for $i \ne n$ and so the above simplifies to
\[
  \sum_{j=1}^{n-1}A^{M_t}(\tau_j,\tau_i)\IP{\tau_j}{e_{n+1}} = 0
\,.
\]
In order to obtain useful information from the above equation we now consider the case of evolving surfaces.
For the remainder of the proof we shall enforce $n=2$.
In this case, we obtain from the above
\[
A^{M_t}(\tau_1,\tau_1)\IP{\tau_1}{e_3} = 0\,.
\]
Now $\IP{\tau_1}{e_3} \ne 0$ since if this were the case then the boundary curve would be parallel to the plane of definition as a graph and in particular at such a point we could not have a new minimum for the quantity $s =\IP{\nu^{M_t}}{e_3}$.
We therefore conclude that $A^{M_t}(\tau_1,\tau_1) = 0$.
Substituting this into \eqref{EQhopfgradient} and using $H \le 0$ we find
\[
0 >
  H\IP{\nu^\Sigma}{e_3}
  - A^{M_t}(\tau_1,\tau_1)\IP{\nu^\Sigma}{e_{n+1}}
=
  H\IP{\nu^\Sigma}{e_3}\geq 0\,.
\]
where we have used the initial condition $\IP{F_0}{e_3} > 0$ which on the Neumann boundary (where $F_0=-\nu^{\Sigma}(F_0)$ for $\Sigma$ a sphere) translates into $\IP{\nu^{\Sigma}}{e_3} < 0 $. This condition is preserved for all times of existence (for example also for the time of a presumed minimum of $s$ on the Neumann boundary) by a comparison principle with the flat plane at zero height, which acts as a barrier for $M_t$. This contradicts the existence of a minimum of $s$ on the Neumann boundary and therefore $s$ is bounded by below a priori by a constant depending only on the initial values.
\end{proof}

\begin{rmk}
The initial condition above forces the height on the Dirichlet boundary away from zero, that is, $h_0\ne0$ in \eqref{immersion}.
\end{rmk}

As outlined at the start of this section, the above is enough to conclude global existence for the mean curvature flow of any initially graphical mean concave surface. The same result is true for mean convex under the initial assumption $\IP{F_0}{e_0}<0$
In order to identify the limit we use a standard argument.

\begin{lem}
\label{LMminimal}
Suppose $M_t = F(M^n,t)$ is a mean curvature flow satisfying \eqref{immersion} with uniformly bounded derivatives of all orders.
If $M_t$ exists globally in time then $M_t$ is asymptotic to a minimal hypersurface.
\end{lem}
\begin{proof}
The flow is a gradient flow for the area functional, and so
\[
\frac{d}{dt}\int_{M^n}\,d\mu = -\int_{M^n}|\partial_tF|^2d\mu
\]
which implies
\[
\int_0^\infty\int_{M^n}H^2d\mu\,dt \le \int_{M^n}\,d\mu\bigg|_{t=0} = c\,.
\]
Since all derivatives are uniformly bounded, we conclude that $M_t\rightarrow M_\infty$ and that the mean curvature of $M_\infty$ is identically zero. This argument has been used before \cite{huisken1989npm,buckland2005mcf,thesisvulcanov}.
\end{proof}

\section*{acknowledgements}
The second author would like to thank Klaus Ecker for his support and supervision during her PhD, when much of the work in this paper was completed.
The second author would also like to thank Gerhard Huisken for his interest and stimulating discussion about the topic of this paper. The first author was supported by Alexander-von-Humboldt Fellowship 1137814 and the second author by DFG Grant ME 3816/2-1. Both authors are currently supported by ARC Discovery Project DP120100097 at the University of Wollongong.
The authors want to especially thank Prof. Ben Andrews for suggesting the use of Killing vector fields in mean curvature flow.
\bibliographystyle{plain}
\bibliography{mbib}

\begin{thebibliography}{10}

\bibitem{altschuler1994}
S.J. Altschuler and L.F. Wu.
\newblock {Translating surfaces of the non-parametric mean curvature flow with
  prescribed contact angle}.
\newblock {\em Calc. Var. Partial Differential Equations}, 2:101--111, 1994.

\bibitem{andrews1994contraction}
B.~Andrews.
\newblock Contraction of convex hypersurfaces in riemannian spaces.
\newblock {\em J. Differential Geom.}, 39(2):407--431, 1994.

\bibitem{buckland2005mcf}
J.A. Buckland.
\newblock {Mean curvature flow with free boundary on smooth hypersurfaces}.
\newblock {\em J. Reine Angew. Math.}, 586:71--90, 2005.

\bibitem{ecker2001local}
K.~Ecker.
\newblock A local monotonicity formula for mean curvature flow.
\newblock {\em Ann. of Math. (2)}, pages 503--525, 2001.

\bibitem{ecker2004rtm}
K.~Ecker.
\newblock {\em {Regularity Theory for Mean Curvature Flow}}.
\newblock Birkhauser, 2004.

\bibitem{ecker1989mce}
K.~Ecker and G.~Huisken.
\newblock {Mean curvature evolution of entire graphs}.
\newblock {\em Ann. of Math. (2)}, 130(2):453--471, 1989.

\bibitem{ecker1991ieh}
K.~Ecker and G.~Huisken.
\newblock {Interior estimates for hypersurfaces moving by mean curvature}.
\newblock {\em Invent. Math.}, 105(1):547--569, 1991.

\bibitem{guan1996mean}
B.~Guan.
\newblock {Mean curvature motion of nonparametric hypersurfaces with contact
  angle condition}.
\newblock {\em Elliptic and parabolic methods in geometry}, page~47, 1996.

\bibitem{huisken1986cch}
G.~Huisken.
\newblock {Contracting convex hypersurfaces in Riemannian manifolds by their
  mean curvature}.
\newblock {\em Invent. Math.}, 84(3):463--480, 1986.

\bibitem{huisken1989npm}
G.~Huisken.
\newblock {Non-parametric mean curvature evolution with boundary conditions}.
\newblock {\em J. Differential Equations}, 77:369--378, 1989.

\bibitem{huisken1990asymptotic}
G.~Huisken.
\newblock Asymptotic behavior for singularities of the mean curvature flow.
\newblock {\em J. Differential Geom.}, 31(1):285--299, 1990.

\bibitem{koeller2007singular}
A.~Koeller.
\newblock {\em {On the Singularity Sets of Minimal Surfaces and a Mean
  Curvature Flow}}.
\newblock PhD thesis, Freie Universit\"at Berlin, 2007.

\bibitem{lieberman1996second}
G.M. Lieberman.
\newblock {\em {Second order parabolic differential equations}}.
\newblock World Scientific Pub. Co. Inc., 1996.

\bibitem{leithesis}
L.~Shahriyari.
\newblock {\em {Translating graphs by mean curvature flow}}.
\newblock PhD thesis, The John Hopkins University, Baltimore, Maryland, USA,
  2012.

\bibitem{thesisstahl}
A.~Stahl.
\newblock {\em {\"Uber den mittleren Kr\"ummungsfluss mit Neumannrandwerten auf
  glatten Hyperfl\"achen}}.
\newblock PhD thesis, Fachbereich Mathematik, Eberhard-Karls-Universit\"at,
  T\"uebingen, Germany, 1994.

\bibitem{stahl1996convergence}
A.~Stahl.
\newblock Convergence of solutions to the mean curvature flow with a neumann
  boundary condition.
\newblock {\em Calc. Var. Partial Differential Equations}, 4(5):421--441, 1996.

\bibitem{stahl1996res}
A.~Stahl.
\newblock {Regularity estimates for solutions to the mean curvature flow with a
  Neumann boundary condition}.
\newblock {\em Calc. Var. Partial Differential Equations}, 4(4):385--407, 1996.

\bibitem{thesisvulcanov}
V.-M. Vulcanov.
\newblock {\em {Mean curvature flow of graphs with free boundaries}}.
\newblock PhD thesis, Freie Universit\"at, Fachbereich Mathematik und
  Informatik, Berlin, Germany, 2010.

\bibitem{vmwheeler2012rotsym}
V.-M. Wheeler.
\newblock Non-parametric radially symmetric mean curvature flow with a free
  boundary.
\newblock {\em Math. Z.}, 276(1-2):281--298, 2014.

\bibitem{wheelermean}
Valentina~Mira Wheeler.
\newblock Mean curvature flow of entire graphs in a half-space with a free
  boundary.
\newblock {\em Journal f{\"u}r die reine und angewandte Mathematik}.

\end{thebibliography}

\end{document}